\documentclass[12pt, reqno]{amsart}
\usepackage{amsmath, amsthm, amscd, amsfonts, amssymb, graphicx, color}
\usepackage[hypertex,bookmarksnumbered, colorlinks, plainpages]{hyperref}
\hypersetup{colorlinks=true,linkcolor=red, anchorcolor=green, citecolor=cyan, urlcolor=red, filecolor=magenta, pdftoolbar=true}
\newtheorem{theorem}{Theorem}[section]
\newtheorem{prop}[theorem]{Proposition}
\newtheorem{lemma}[theorem]{Lemma}
\newtheorem{cor}[theorem]{Corollary}
\theoremstyle{definition}
\newtheorem{defn}[theorem]{Definition}
\newtheorem{eg}[theorem]{Example}
\newtheorem{note}[theorem]{Note}

\theoremstyle{remark}
\newtheorem{rem}[theorem]{Remark}

\numberwithin{equation}{section}

\begin{document}

\title{OPERATORS THAT ATTAIN REDUCED MINIMUM}
\author[S. H. Kulkarni, \MakeLowercase{and} G. Ramesh]{S. H. Kulkarni,$^1$  \MakeLowercase{and} G. Ramesh$^2$$^{*}$}

\address{$^{1}$ Department of Mathematics, Indian Institute of Technology - Madras, Chennai 600 036.}
\email{shk@iitm.ac.in}

\address{$^{2}$ Department of Mathematics, Indian Institute of Technology - Hyderabad, Kandi, Sangareddy, Telangana, India 502 285.}
\email{rameshg@iith.ac.in}
\date{\currenttime \;  \today}


\date{Received: xxxxxx; Revised: yyyyyy; Accepted: zzzzzz.
\newline \indent $^{*}$Corresponding author}

\begin{abstract}
Let $H_1, H_2$ be complex Hilbert spaces and $T$ be a densely defined closed linear operator from
its domain $D(T)$, a dense subspace of $H_1$, into $H_2$. Let
$N(T)$ denote the null space of $T$ and $R(T)$ denote the range of $T$.

Recall that $C(T) := D(T) \cap N(T)^{\perp}$  is called the {\it carrier space of} $T$ and
the {\it reduced minimum modulus } $\gamma(T)$ of $T$ is defined as:
$$  \gamma(T) := \inf \{\|T(x)\| : x \in C(T), \|x\| = 1 \} .$$

 Further, we say that  $T$
 {\it  attains its reduced minimum  modulus} if there exists $x_0 \in C(T) $ such that   $\|x_0\| = 1$ and  $\|T(x_0)\| =   \gamma(T)$.
 We discuss some properties of operators that attain reduced minimum modulus. In particular, the following results are proved.

\begin{enumerate}

\item
The operator  $T$ attains its reduced minimum modulus if and
only if its Moore-Penrose inverse $T^{\dagger}$ is bounded and attains its norm, that is, there exists $y_0 \in H_2$
such that $\|y_0\| = 1$ and $\|T^{\dagger}\| = \|T^{\dagger}(y_0) \|$.
\item
For each $\epsilon > 0$, there exists a bounded operator $S$ such that $\|S\| \leq \epsilon$ and $T + S$ attains its
reduced minimum.

\end{enumerate}
\end{abstract}

\maketitle

\section{Introduction }

Let $H_1$ and $H_2$ be complex Hilbert spaces and $T:H_1\rightarrow H_2$ be a bounded linear operator. We say $T$ to be {\it norm attaining} if there exists $x_0\in H_1$ such that $\|x_0\|=1$ and $\|Tx_0\|=\|T\|$. The norm attaining operators are well studied in the literature by several authors (see \cite{shkarin} for details and references there in). A well known theorem in this connection is the Lindestrauss theorem which asserts the denseness of norm attaining operators in the space of bounded linear operators between two Hilbert spaces with respect to the operator norm (see for example, \cite{enfloetal} for a simple proof of this fact).

A natural analogue for this class of operators is the class of minimum attaining operators. Recall that a bounded  operator $T:H_1\rightarrow H_2$ is said to be {\it minimum attaining}, if there exists $x_0\in H_1$ with $\|x_0\|=1$ such that
$\|Tx_0\|=m(T)$, the {\it minimum modulus} of $T$. This class of operators was first introduced by Carvajal and Neves in \cite{carvajalneves2} and several basic properties were also studied in the line of norm attaining operators.

 A Lindenstrauss type theorem for  minimum attaining operators is proved in \cite{shkgrminattaining}. Moreover, rank one perturbations of closed operators is also discussed.

In this article, we define operators that attain the reduced minimum  modulus and establish several basic properties of such operators. We prove that if a densely defined closed operator $T$ attains its reduced minimum, then its Moore-Penrose inverse $T^{\dagger}$ is bounded and attains its norm. It turns out that this class is a subclass of minimum attaining operators as well as the class of closed range operators. Finally, we observe that this class is  dense in the class of densely defined closed operators with respect to the gap metric as well as with respect to the carrier graph topology (see \cite{shkgrcgt} for details). We prove several consequences of this result.

 In the second section we summarize without proofs the relevant material on densely defined closed operators, the gap metric and the carrier graph topology. In the third section we define the reduced minimum attaining operators, prove some of the basic and important properties of such operators and compare with those of minimum attaining operators. In proving most of our results, we make use of the corresponding result for minimum attaining operators, which can be found in \cite{shkgrminattaining} and \cite{shkgr7}. In Section four we give a correct formula for the distance between a symmetric operator and a scalar multiple of the identity operator in the metric defined in \cite{shkgrcgt}.

\section{Preliminaries}

Through out we consider infinite dimensional complex Hilbert spaces which will be denoted by $H, H_1,H_2$ etc. The inner product and the
induced norm are denoted  by  $\langle \cdot\rangle$ and
$||.||$, respectively.  The closure of a subspace $M$ of $H$ is denoted by $\overline{M}$. We denote the unit sphere of $M$ by $S_M={\{x\in M:\|x\|=1}\}$.

Let $T$ be a linear operator with domain $D(T)$, a subspace of $H_1$ and taking values in $H_2$. If $D(T)$ is dense in $H_1$, then $T$ is called a densely defined operator.

The  graph $G(T)$ of $T$ is defined by  $G(T):={\{(Tx,x):x\in D(T)}\}\subseteq H_1\times H_2$. If $G(T)$ is
closed, then $T$ is called a closed operator.  Equivalently, $T$ is closed if and only if  if $(x_n)$ is a sequence in  $D(T)$  such that  $x_n\rightarrow x\in H_1$ and $Tx_n\rightarrow y\in H_2$, then $x\in D(T)$ and $Tx=y$.

For a densely defined operator, there exists a unique linear operator (in fact, a closed operator) $T^*:D(T^*)\rightarrow H_1$, with
\begin{equation*}
D(T^*):={\{y\in H_2: x\rightarrow \langle Tx,y\rangle \, \text{for all}\, x\in D(T)\,\text{is continuous}}\}\subseteq H_2
\end{equation*}
 satisfying $\langle Tx,y\rangle =\langle x,T^*y\rangle$ for all $x\in D(T)$ and $y\in D(T^*)$.
We say $T$ to be bounded if there exists $M>0$ such that $\|Tx\|\leq M\|x\|$ for all $x\in D(T)$. Note that if $T$ is densely defined and bounded then $T$ can be extended to all of $H_1$ in a unique way.

By the closed graph Theorem \cite{rud}, an everywhere defined
closed operator is bounded.  Hence the domain of an unbounded closed operator is a proper subspace of a Hilbert space.

The space of all bounded linear operators between $H_1$ and $H_2$ is
denoted by $\mathcal B(H_1,H_2)$ and the class of all densely defined, closed linear operators between $H_1$ and $H_2$ is denoted by $\mathcal C(H_1,H_2)$. We write $\mathcal B(H,H)=\mathcal B(H)$  and $\mathcal C(H,H)=\mathcal C(H)$.

If $T\in \mathcal C(H_1,H_2)$, then  the null space and the range space of $T$ are denoted by $N(T)$ and $R(T)$ respectively and the space $C(T):=D(T)\cap N(T)^\bot$ is called the carrier of $T$.
In fact, $D(T)=N(T)\oplus^\bot C(T)$ \cite[page 340]{ben}.

 Let $T_C:=T|_{C(T)}$.  As $\overline{C(T)}=N(T)^{\bot}$ (see \cite[Lemma 3.3]{shkgrcgt} for details),  $T\in \mathcal C(N(T)^{\bot},H_2)$.

 Let $S,T\in \mathcal C(H)$ be operators with domains $D(S)$ and $D(T)$, respectively. Then $S+T$ is an operator with domain $D(S+T)=D(S)\cap D(T)$ defined by $(S+T)(x)=Sx+Tx$ for all $x\in D(S+T)$.
 The operator $ST$  has the domain $D(ST)={\{x\in D(T): Tx\in D(S)}\}$ and is defined as $(ST)(x)=S(Tx)$ for all $x\in D(ST)$.

If $S$ and $T$ are closed operators  with the property that $D(T)\subseteq D(S)$ and $Tx=Sx$ for all $x\in D(T)$, then $T$ is called the restriction of $S$ and $S$ is called an extension of $T$. We denote this by $T\subseteq S$.

An operator $T\in \mathcal C(H)$ is said to be
\begin{enumerate}
 \item normal if $T^*T=TT^*$
  \item self-adjoint if $T=T^*$
  \item symmetric if $T\subseteq T^*$
 \item positive if  $T=T^*$ and $\langle Tx,x\rangle \geq 0$ for all $x\in D(T)$.
\end{enumerate}

Let $V\in \mathcal B(H_1,H_2)$. Then $V$ is called
\begin{enumerate}
\item an isometry if $\|Vx\|=\|x\|$ for all $x\in H_1$
\item a partial isometry if $V|_{N(V)^{\bot}}$ is an isometry. The space $N(V)^{\bot}$ is called the initial space or the initial domain and the space $R(V)$ is called the final space or the final domain of $V$.
\end{enumerate}

If $M$ is a closed subspace of a Hilbert space $H$, then $P_M$ denotes the orthogonal projection $P_M:H\rightarrow H$ with range $M$,  and $S_M$ denotes the unit sphere of $M$.

Here we recall definition and properties of the Moore-Penrose inverse (or generalized inverse) of a densely defined closed operator that we need for our purpose.
 \begin{defn}\label{geninv}(Moore-Penrose Inverse)\cite[Pages 314, 318-320]{ben}
Let $T\in \mathcal C(H_1,H_2)$. Then there exists
a unique  operator $T^\dagger \in \mathcal
C(H_2,H_1)$ with domain $D(T^\dagger)=R(T)\oplus ^\bot R(T)^\bot$
and has the following properties:
\begin{enumerate}
\item $TT^\dagger y=P_{\overline{R(T)}}~y, ~\text{for all}~y\in D(T^\dagger)$

\item $T^\dagger Tx=P_{N(T)^\bot} ~x, ~\text{for all}~x\in D(T)$

\item $N(T^\dagger)=R(T)^\bot$.
\end{enumerate}
This unique operator $T^\dagger$ is called the \textit{Moore-Penrose inverse} or the \textit{generalized inverse}  of $T$.\\
The following property of $T^\dagger$ is also well known. \noindent
For every $y\in D(T^\dagger)$, let $$L(y):=\Big\{x\in D(T):
||Tx-y||\leq ||Tu-y||\quad \text{for all} \quad u\in D(T)\Big\}.$$
 Here any $u\in L(y)$ is called a \textit{least square solution} of the operator equation $Tx=y$. The vector  $T^\dagger y\in L(y),\,||T^\dagger y||\leq ||x||\quad \text{for all} \quad x\in L(y)$
 and it is called the  \textit{least square solution of minimal norm}.
 A different treatment of $T^\dagger$ is given in \cite[Pages 336, 339, 341]{ben},
 where it is called ``\textit{the Maximal Tseng generalized Inverse}".
\end{defn}
\begin{theorem}\cite[Page 320]{ben}\label{propertiesofgeninve1}
Let $T\in \mathcal C(H_1,H_2)$. Then
\begin{enumerate}
\item $D(T^\dagger)=R(T)\oplus^\bot R(T)^\bot, \quad
N(T^\dagger)=R(T)^\bot=N(T^*)$
\item $R(T^\dagger)=C(T)$
\item  $ T^\dagger \in \mathcal C(H_2,H_1)$
\item $T^\dagger$ is continuous if and only $R(T)$ is closed
\item $T^{\dagger \dagger}=T$
\item $T^{* \dagger}=T^{\dagger *}$
\item $N(T^{* \dagger})=N(T)$
\item $T^*T$ and $T^\dagger T^{* \dagger}$ are positive and $(T^*T)^\dagger =T^\dagger T^{*
\dagger}$
\item $TT^*$ and $ T^{* \dagger}T^\dagger$ are positive and $(TT^*)^\dagger= T^{*
\dagger}T^\dagger$.
\end{enumerate}
\end{theorem}

\begin{defn}\label{normattainingdefn}
Let $T\in \mathcal B(H_1,H_2)$. Then $T$ is said to be {\it norm attaining} if there exists $x_0\in S_{H_1}$ such that $\|Tx_0\|=\|T\|$.
\end{defn}
We denote the set of all norm attaining operators between $H_1,H_2$ by $\mathcal N(H_1,H_2)$ and $\mathcal N(H,H)$ by $\mathcal N(H)$.
\begin{defn}\cite{ben, goldberg,taylorminmmod}\label{minmmodulus}
 Let $T\in \mathcal C(H_1,H_2)$. Then
 \begin{align*}
 m(T)&:=\inf{\{\|Tx\|: x\in S_{D(T)}}\}\\
 \gamma(T)&:=\inf{\{\|Tx\|:x\in S_{C(T)}}\},
 \end{align*}
 are  called the {\it minimum modulus} and the {\it reduced minimum modulus} of $T$, respectively. The operator $T$ is said to be bounded  below if and only if $m(T)>0$.
\end{defn}

\begin{rem}
If $T\in \mathcal C(H_1,H_2)$, then
\begin{itemize}
\item[(a)] $m(T)\leq \gamma(T)$ and equality holds if $T$ is one-to-one
\item[(b)] $m(T)>0$ if and only if $R(T)$ is closed and $T$ is one-to-one
    \end{itemize}
\end{rem}

\begin{prop}\label{equivalentclosedrange}\cite{ben,kato}
Let $T\in \mathcal C(H_1,H_2)$. Then the following statements are equivalent;
\begin{enumerate}
\item $R(T)$ is closed
\item $R(T^*)$ is closed
\item $T_0:=T|_{C(T)}$ has a bounded inverse
\item $\gamma(T)>0$
\item \label{minmodreciprocal}$T^{\dagger}$ is bounded. In fact, $\|T^{\dagger}\|=\frac{1}{\gamma(T)}$
\item $R(T^*T)$ is closed
\item $R(TT^*)$ is closed.
\end{enumerate}
\end{prop}
\begin{rem}\label{reciprocalinversenorm}
 If $T\in \mathcal C(H)$ and $T^{-1}\in \mathcal B(H)$, then $m(T)=\frac{1}{\|T^{-1}\|}$, by (\ref{minmodreciprocal}) of Proposition \ref{equivalentclosedrange}.

\end{rem}
\begin{theorem}\label{squareroot}\cite[theorem 13.31, page 349]{rud}\cite[Theorem 4, page 144]{birmannsolomyak}
Let $T\in \mathcal C(H)$ be positive. Then there exists a unique positive operator $S$ such that
$T=S^2$. The operator $S$ is called the square root of $T$ and is denoted by $S=T^\frac{1}{2}$.
\end{theorem}

\begin{theorem}\cite[Theorem 2, page 184]{birmannsolomyak}
Let $T\in \mathcal C(H_1,H_2)$. Then there exists a unique partial isometry $V:H_1\rightarrow H_2$ with initial space $\overline{R(T^*)}$ and range $\overline{R(T)}$ such that $T=V|T|$.
\end{theorem}

\begin{rem}
For $T\in \mathcal C(H_1,H_2)$, the operator  $|T|:=(T^*T)^\frac{1}{2}$ is called the modulus of $T$. Moreover,  $D(|T|)=D(T),\; N(|T|)=N(T)$ and $\overline{R(|T|)}=\overline{R(T^*)}$.
 As $\|Tx\|=\||T|x\|$ for all $x\in D(T)$,  we can conclude that  $m(T)=m(|T|)$, and $\gamma(T)=\gamma(|T|)$.
 \end{rem}
\begin{defn}\cite[page 346]{rud}
 Let $T\in \mathcal C(H)$. The resolvent of $T$ is defined by
 \begin{equation*}
 \rho(T):={\{\lambda \in \mathbb C: T-\lambda I:D(T)\rightarrow H\; \text{is invertible and}\; (T-\lambda I)^{-1}\in \mathcal B(H)}\}
 \end{equation*}
and
\begin{align*}
\sigma(T):&=\mathbb C \setminus \rho(T)\\
\sigma_p(T):&={\{\lambda \in \mathbb C: T-\lambda I:D(T)\rightarrow H \; \text{is not one-to-one}}\},
\end{align*}
are called the  spectrum and the point spectrum of $T$, respectively.
\end{defn}

 \begin{defn}\cite[Page 267]{kato}
 Let $T\in \mathcal C(H)$. Then
 the \textit{numerical range} of $T$ is defined by
\begin{equation*}
W(T):=\Big\{\langle Tx,x\rangle: x\in S_{D(T)}\Big\}.
\end{equation*}

 \end{defn}
The following Proposition is proved in \cite[Chapter 10]{lance} for regular (unbounded) operators between Hilbert $C^*$-modules, which is obviously true for densely defined closed operators in a Hilbert space.
\begin{prop}\cite[Lemma 5.8]{schmudgen} Let $T\in \mathcal C(H)$.  Let $Q_T:=(I+T^*T)^{-\frac{1}{2}}$ and $F_T:=TQ_T$. Then
\begin{enumerate}
\item $Q_T\in \mathcal B(H)$ and $0\leq Q_T\leq I$
\item $R(Q_T)=D(T)$
\item $(F_T)^*=F_{T^*}$
\item $\|F_T\|<1$ if and only if $T\in \mathcal B(H_1,H_2)$
\item $T=F_T(I-F_T^*F_T)^{-\frac{1}{2}}$
\item $Q_T=(I-F_T^*F_T)^{\frac{1}{2}}$.
\end{enumerate}
The operator $F_T$ is called the bounded transform of $T$ or the $z$-transform of $T$.
\end{prop}

\begin{lemma}\cite{schock,gr92,ped}\label{boundedclosure}
 Let $T\in \mathcal C(H_1,H_2)$. Denote $\check T=(I+T^*T)^{-1}$ and $\widehat T=(I+TT^*)^{-1}$. Then
\begin{enumerate}
\item $\check T\in \mathcal B(H_1)$, $\widehat T\in \mathcal
B(H_2)$
\item $\widehat TT\subseteq T\check T$,\quad $||T\check T||\leq \displaystyle
\frac{1}{2}$ and $\check TT^*\subseteq T^*\widehat T$,\quad
$||T^*\widehat T||\leq \displaystyle \frac{1}{2}$.

\end{enumerate}
\end{lemma}

One of the most useful and well studied metric on $C(H_1,H_2)$ is the gap metric. Here we give some details.
\begin{defn}[Gap between subspaces]\cite[page 197]{kato}
Let $H$ be a Hilbert space and $M,N$ be closed subspaces of $H$. Let
$P=P_M$ and $Q=P_N$. Then the gap between $M$ and $N$ is defined by
\begin{equation*}
\theta(M, N)=\|P-Q\|.
\end{equation*}
\end{defn}

If $S,T \in \mathcal C(H_1,H_2)$, then  $G(T),G(S)\subseteq
H_1\times H_2$ are closed subspaces. The gap between $G(T)$ and $G(S)$ is called the gap between $T$ and $S$.
For a deeper discussion on these concepts we refer to \cite[Chapter
IV]{kato} and \cite[page 70]{akh}.

We have the following formula for the gap between two closed operators;

\begin{theorem}\cite{shkgrgap}\label{shkgrgapformula}
  Let $S,T\in \mathcal C(H_1,H_2)$. Then the operators $\widehat{T}^{\frac{1}{2}}S\check{S}^{\frac{1}{2}}$, $ T\check{T}^{\frac{1}{2}}\check{S}^{\frac{1}{2}}$,  $S\check{S}^{\frac{1}{2}}\check{T}^{\frac{1}{2}}$ and $\widehat{S}^{\frac{1}{2}}T\check{T}^{\frac{1}{2}}$ are bounded and
  \begin{equation}
  \theta(S,T)=\max \Big\{\left  \|   T\check{T}^{\frac{1}{2}}\check{S}^{\frac{1}{2}}-\widehat{T}^{\frac{1}{2}}S\check{S}^{\frac{1}{2}}\right \|, \left \|S\check{S}^{\frac{1}{2}}\check{T}^{\frac{1}{2}}-\widehat{S}^{\frac{1}{2}}T\check{T}^{\frac{1}{2}}\right \|\Big\}.
  \end{equation}
\end{theorem}

  On $\mathcal B(H)$, the norm topology and the topology induced by the gap metric are the same. This can be seen from the following inequalities.

  \begin{theorem}\cite[Theorem 2.5]{nakamotogapformulas}\label{equalityoftopsnakamoto}
  Let $A,B\in \mathcal B(H)$. Then

  \begin{equation*}
  \theta(A,B)\leq \|A-B\|\leq \sqrt{1+\|A\|^2}\; \sqrt{1+\|B\|^2}\; \theta(A,B).
  \end{equation*}

  \end{theorem}
    We remark that though the above result is stated for operators defined on a Hilbert space, it remains true for operators defined between two different Hilbert spaces.

\begin{defn}\label{cgt}
Let $T\in \mathcal C(H_1,H_2)$. Define  the
\textit{Carrier Graph of $T$} by $$G_{C}(T):={\{(x,Tx):x\in
C(T)}\}\subseteq H_1\times H_2.$$
 For $S,T\in \mathcal C(H_1,H_2)$, the gap between $G_{C}(S)$ and $G_{C}(T)$ is denoted by,
  \begin{equation*}
  \eta(S,T)=\|P_{G_C(S)}-P_{G_C(T)}\|.
  \end{equation*}
  The  topology induced by the metric $\eta (\cdot,\cdot)$ on $\mathcal C(H_1,H_2)$ is called the \textit{Carrier Graph Topology}.

  To compute the $\eta(S,T)$ we can use the following formula;

  \begin{theorem}\label{equivalenceofcgtandgap}
Let $T,S\in \mathcal C(H_1,H_2)$. Then
$$|\eta(T,S)-\theta (N(T),N(S))|\leq \theta (T,S)\leq \eta(T,S)+\theta(N(T),N(S)).$$
\end{theorem}
\end{defn}

If $N(T)=N(S)$, by Theorem \ref{equivalenceofcgtandgap}, we can conclude that $\eta(S,T)=\theta(S,T)$.
For the details of this metric we refer to \cite{shkgrcgt}.


\section{Main Results}
In this section we define  reduced minimum modulus  attaining operators and discuss their properties. Recall that $T\in \mathcal C(H_1,H_2)$ is called minimum attaining if there exists $x_0\in S_{D(T)}$ such that
$\|Tx_0\|=m(T)$. In particular, if $T\in \mathcal B(H_1,H_2)$, then $T$ is minimum attaining if there exists $x_0\in S_{H_1}$ such that $\|Tx_0\|=m(T)$.

We denote the class of minimum attaining  densely defined closed operators between $H_1$ and $H_2$ by $\mathcal M_c(H_1,H_2)$ and $\mathcal M_c(H,H)$ by $\mathcal M_c(H)$. The class of bounded minimum attaining operators is denoted by $\mathcal M(H_1,H_2)$ and $\mathcal M(H, H)$ by $\mathcal M(H)$.

We propose the following definition;
\begin{defn}
We say $T\in \mathcal C(H_1,H_2)$ to be  {\it reduced  minimum attaining} if there exists $x_0\in S_{C(T)}$ such that $\|Tx_0\|=\gamma(T)$.
\end{defn}

 The class of reduced minimum attaining densely defined  closed linear operators between $H_1$ and $H_2$ is denoted  by $\Gamma_c(H_1,H_2)$. If $H_1=H_2=H$, then we write $\Gamma_c(H_1,H_2)$ by $\Gamma_c(H)$.
 The class of bounded operators which attain the reduced minimum is denoted by $\Gamma(H_1,H_2)$ and $\Gamma(H,H)$ is denoted by $\Gamma(H)$.
\begin{theorem} \label{dualrelation}
Let $T\in \mathcal C(H_1,H_2)$. Then $T$ attains its reduced minimum if and
only if $T^{\dagger}$ is bounded and attains its norm.
\end{theorem}
\begin{proof}
Suppose $T$ attains its reduced minimum. Then there exists $x_0 \in C(T) $ such that   $\|x_0\| = 1$  and $\|T(x_0)\| =   \gamma(T)$.
We must have $\gamma(T) > 0$ as otherwise $x_0 \in N(T) $ will imply $x_0 = 0$, a contradiction. This implies that $T^{\dagger}$ is bounded.
Let $y_0 = T(x_0)/\|T(x_0)\| = T(x_0)/\gamma(T)$. Then   $\| y_0\| = 1$ and
\begin{equation*}
\| T^{\dagger}(y_0)\| = \|T^{\dagger}T(x_0)\|/\gamma(T)
= \|x_0\|/\gamma(T) = 1/\gamma(T) = \|T^{\dagger}\|.
\end{equation*}
Thus $T^{\dagger}$ attains its norm.

Conversely assume that  $T^{\dagger}$ is bounded and attains its norm. Then there exists $y_0 \in H_2$ such that $\|y_0\| = 1$
and $\| T^{\dagger}(y_0)\| = \|T^{\dagger}\|$. Let $y_0 = u + v$ where $u \in R(T)$ and $v \in R(T)^{\perp}$. Suppose $v \neq 0$.
Then $\|u\| < 1$. Hence
\begin{equation*}
\| T^{\dagger}\| = \| T^{\dagger}(y_0)\| = \| T^{\dagger}(u)\|  \leq  \|T^{\dagger}\| \|u\| < \| T^{\dagger}\|,
\end{equation*}
a contradiction. This implies that $v = 0$, hence $ y_0 \in R(T)$. Thus there exists $x_0 \in C(T)$ such that $y_0 = T(x_0)$. Then
$x _0 = T^{\dagger}(y_0)$, hence
$\|x _0\| = \| T^{\dagger} \|$.  Let
$z_0 = x_0/ \| x_0\|$. Then $z_0 \in C(T)$,  $\|z_0\| = 1$ and
\begin{equation*}
\|T(z_0)\| = \|y_0\| / \|x_0\| = 1/ \|T^{\dagger}\| = \gamma(T) .
\end{equation*}
Thus $T$ attains the reduced minimum modulus.
\end{proof}

\begin{cor}
Let $T\in \mathcal C(H_1,H_2)$. Suppose $T$ is one-to-one. Then the following are equivalent.
\begin{enumerate}
\item  \label{minattaining}  $T\in \mathcal M_c(H_1,H_2)$
\item  \label{rmattaining}$T\in \Gamma_c(H_1,H_2)$
\item \label{mpibdd}$T^{\dagger} \in \mathcal B(H_2,H_1)$ and attains its norm.
\end{enumerate}
\end{cor}
\begin{proof}
Since $T$ is injective, $N(T)={\{0}\}$. Hence $C(T)=D(T)$ and $\gamma(T)=m(T)$. This shows equivalence of (\ref{minattaining}) and (\ref{rmattaining}). Equivalence of  (\ref{rmattaining}) and (\ref{mpibdd}) follows  from Theorem \ref{dualrelation}.
\end{proof}
\begin{lemma}\label{mmandrmm}
 Let $T\in \mathcal C(H_1,H_2)$. Then $T\in \Gamma_c(H_1,H_2)$ if and only if $T_C\in \mathcal M_c(N(T)^{\bot},H_2)$.
\end{lemma}
\begin{proof}
The proof follows from the fact that $m(T_C)=\gamma(T)$.
\end{proof}

\begin{prop}\label{closedrange}
Let $T\in \Gamma_c(H_1,H_2)$. Then $R(T)$ is closed.
\end{prop}

\begin{proof}
 This follows from Theorem \ref{dualrelation} and Proposition \ref{equivalentclosedrange}.
\end{proof}
\begin{eg}
\begin{enumerate}
\item All orthogonal projections on a Hilbert space attain their reduced minimum
\item An operator with non closed range cannot attain its reduced minimum.
 \end{enumerate}
\end{eg}

\begin{prop}
Let $T\in \mathcal C(H_1,H_2)$. Then  $T\in \Gamma_c(H_1,H_2)$ if and only if  $|T|\in \Gamma_c(H_1)$.
\end{prop}
\begin{proof}
By definition $D(|T|)=D(T)$ and  $N(|T|)=N(T)$. Hence $C(|T|)=C(T)$. Also, $\|Tx\|=\||T|x\|$ for all $x\in D(T)$. Thus $T\in \Gamma(H_1,H_2)$ if and only if $|T|\in \Gamma(H_1)$.
\end{proof}

\begin{prop}\label{rmmselfadjoint}\cite[Proposition 4.2]{knr}
Let $T=T^*\in \mathcal C(H)$. Then $\gamma(T)=d(0,\sigma(T)\setminus {\{0}\})$.
\end{prop}
\begin{lemma}\label{minmodselfadj}
\begin{enumerate}
\item \label{minmodformula} Let $T\in \mathcal C(H)$ be  self-adjoint. Then  $m(T)=d(0,\sigma(T))$
\item \label{minmodspectralpt} If $T\in \mathcal C(H_1,H_2)$, then $m(T)\in \sigma(|T|)$. In particular, if $H_1=H_2=H$ and  $T\geq 0$, then $m(T)\in \sigma(T)$.
\end{enumerate}
\end{lemma}
\begin{proof}
Proof of  (\ref{minmodformula}): If $T$ is not invertible, then $0\in \sigma(T)$ and $T$ is not bounded below. Hence in this case $m(T)=0=d(0,\sigma(T))$.

           Next assume that $0\notin \sigma(T)$.  Since $\sigma(T)$ is closed (\cite[Proposition 2.6, Page 29]{schmudgen}), we can conclude that $d(0,\sigma(T))>0$. Also, as $T^{-1}\in \mathcal B(H)$, $T$ must be bounded below. Hence $m(T)>0$. In this case, $m(T)=\gamma(T)$. Now, by Proposition \ref{rmmselfadjoint},  we have $m(T)=\gamma(T)=d(0,\sigma(T)\setminus {\{0}\})=d(0,\sigma(T))$.

Proof of (\ref{minmodspectralpt}): Note that $|T|\geq 0$ and by (\ref{minmodformula}), we have that $m(T)=m(|T|)=d(0,\sigma(|T|))$. Since, $\sigma(|T|)$ is closed, we can conclude that $m(T)\in \sigma(|T|)$.
 If $H_1=H_2=H$ and $T\geq 0$, then we have $|T|=T$. Hence in this case the result follows.
\end{proof}

\begin{rem}
Let $T\in \mathcal C(H)$ be normal. Then we can prove the formula $m(T)=d(0,\sigma(T))$. First note that the crucial point in proving this in the self-adjoint case is Proposition \ref{rmmselfadjoint}. This is proved for normal operators
in \cite[Theorem 4.4.5]{GRthesis}. Now following along the similar lines of Proposition \ref{minmodselfadj}, we can obtain the formula.
\end{rem}

\begin{prop}
Let $T=T^*\in \mathcal C(H)$. Then $T\in \Gamma_c(H)$ if and only if either $\gamma(T)$ or $-\gamma(T)$ is an eigenvalue of $T$. In particular, if $T\geq 0$, then $T\in \Gamma_c(H)$ if and only if $\gamma(T)$ is an eigenvalue of $T$.
\end{prop}
\begin{proof}
We have by Lemma \ref{mmandrmm},  that $T\in \Gamma_c(H)$ if and only if $T_C\in \mathcal M(N(T)^{\bot})$. As $N(T)^{\bot}$ is a reducing subspace for $T$, $T_C$ is self-adjoint. Now, $T_C\in \mathcal M(N(T)^{\bot})$ if and only
either $m(T_C)$ or $-m(T_C)$ is an eigenvalue for $T_C$ and hence for $T$. Since $m(T_C)=\gamma(T)$, the conclusion follows. In particular, if $T\geq 0$, the eigenvalues of $T$ are positive, so we can conclude that $T\in \Gamma_c(H)$ if and only if
$\gamma(T)\in \sigma_p(T)$.
\end{proof}

\begin{prop}\label{redmingram}
Let $T\in \mathcal C(H_1,H_2)$. Then $T\in \Gamma_c(H_1,H_2)$ if and only if $T^*T\in \Gamma_c(H_1)$.
\end{prop}

\begin{proof}
By Theorem \ref{dualrelation}, $T\in \Gamma_c(H_1,H_2)$ if and only if $R(T)$ is closed and $T^{\dagger}\in \mathcal N(H_2,H_1)$. This is equivalent to the condition that $(T^*)^{\dagger}=(T^{\dagger})^* \in \mathcal N(H_1,H_2)$. This is in turn equivalent to the fact that  $(T^{\dagger})(T^{\dagger})^*\in \mathcal N(H_2)$. But $(T^{\dagger})(T^{\dagger})^*=(T^*T)^{\dagger}$, by Theorem \ref{propertiesofgeninve1}. Thus by Theorem \ref{dualrelation}, $T^*T\in \Gamma_c(H_1)$.
\end{proof}
\begin{prop}
Let $T\in \mathcal C(H_1,H_2)$. Then $T\in \Gamma_c(H_1,H_2)$ if and only if $T^*\in \Gamma_c(H_2,H_1)$.
\end{prop}
\begin{proof}
If $T\in \Gamma_c(H_1,H_2)$, then $R(T)$ is closed and so is $R(T^*)$. Also, we have $\gamma(T)=\gamma(T^*)$. By Theorem \ref{dualrelation}, $T^{\dagger}\in \mathcal N(H_2,H_1)$. Also, $(T^{\dagger})^{*}\in \mathcal N(H_1,H_2)$, by  \cite[Proposition 2.5]{carvajalneves1}. Note that $(T^{\dagger})^*=(T^*)^{\dagger}$, by Thoerem \ref{propertiesofgeninve1}. Hence by Theorem \ref{dualrelation} again, $T^*\in \Gamma_c(H_2,H_1)$. Applying the same result for $T^*$ and observing that $T^{**}=T$, we get the other way implication.
\end{proof}

\begin{rem}
The above result need not hold for minimum attaining operators. Let $H=\ell^2$ and ${\{e_n:n\in \mathbb N}\}$ denote the standard orthonormal basis for $H$. That is $e_n(m)=\delta_{nm}$, the Dirac delta function.
Define operators $D,R:H\rightarrow H$ by
\begin{align*}
De_n&=\dfrac{1}{n}e_n,\\
Re_n&=e_{n+1} \; \text{for each}\;  n\in \mathbb N.
\end{align*}
Let $T=RD$. Since $R$ is an isometry, we have $m(T)=m(D)=d(0,\sigma(D))=\inf{\Big\{\dfrac{1}{n}: n\in \mathbb N}\Big\}=0$. Since $0\notin \sigma_p(D)$, $D$ is not minimum attaining. Thus $T$ is not minimum attaining. But, $N(T^*)=\text{span}{\{e_1}\}$. So $m(T^*)=0$ and $T^*\in \mathcal M(H)$. Note that $T^*T=D^2$ cannot have closed range since $D$ is compact. Equivalently, $R(T)$ is not closed, whence $T$ cannot attain its reduced minimum by Theorem \ref{dualrelation}.
\end{rem}

\begin{prop}\label{subclassofminatt}
Let $T\in \mathcal C(H_1,H_2)$. If $T\in \Gamma_c(H_1,H_2)$, then $T\in \mathcal M_c(H_1,H_2)$.
\end{prop}
\begin{proof}
First assume that $T$ is one-to-one. Then $\gamma(T)=m(T)$. Hence if $T\in \Gamma_c(H_1,H_2)$, then clearly $T\in \mathcal M_c(H_1,H_2)$. If $T$ is not one-to-one, then $N(T)\neq {\{0}\}$. Hence in this case
$m(T)=0$ and there exists a $0\neq x\in N(T)$ such that $Tx=0$. Hence clearly $T\in \mathcal M_c(H_1,H_2)$. This completes the proof.
\end{proof}

\begin{prop}\label{minmodformulapos}\cite[Proposition 3.5]{shkgr7}
Let $T\in \mathcal C(H)$ be positive. Then
\begin{equation*}
m(T)=\inf{\{\langle Tx,x\rangle: x\in S_{D(T)}}\}.
\end{equation*}
In particular, if $T\in \mathcal C(H_1,H_2)$, then $m(T^*T)=m(T)^2$.
\end{prop}
\begin{prop}\label{rmmformula}
Let $T\in \mathcal C(H)$ be positive. Then
\begin{equation*}
\gamma(T)=\inf{\{\langle Tx,x\rangle: x\in S_{C(T)}}\}.
\end{equation*}
In particular, if $T\in \mathcal C(H_1,H_2)$, then $\gamma(T^*T)=\gamma(T)^2$.
\end{prop}
\begin{proof}
 Since, $T_C$ is positive, we have by Proposition \ref{minmodformulapos},
\begin{align*}
\gamma(T)=m(T_C)&=\inf{\{\langle T_Cx,x\rangle: x\in S_{C(T)}}\}\\
                               &=\inf{\{\langle Tx,x\rangle: x\in S_{C(T)}}\}.\qedhere
\end{align*}

Further, if $T\in \mathcal C(H_1,H_2)$, then $T^*T\in \mathcal C(H_1)$ is positive. Thus by applying the above formula for $T^*T$ and by the definition of $\gamma(T)$, we get the conclusion.
\end{proof}

\begin{rem}
If $T\in \mathcal C(H)$ be positive. Then the following statements are equivalent (see \cite[Proposition 3.8]{shkgr7}):

\begin{enumerate}
 \item $T\in \mathcal M_c(H)$
 \item  $m(T)$ is an eigenvalue of $T$
 \item $m(T)$ is an extreme point of $W(T)$.
\end{enumerate}

In general  if $T\in \Gamma_c(H)$ and is positive, then $\gamma(T)$ need not be an extreme point of the numerical range of $T$.
To see this,  consider the operator $T$ on $\mathbb C^3$,  whose matrix with respect to the standard orthonormal basis of $\mathbb C^3$ is $\left(
                                                                            \begin{array}{ccc}
                                                                              0 & 0 & 0 \\
                                                                              0 & \frac{1}{2} & 0 \\
                                                                              0 & 0 & 1 \\
                                                                            \end{array}
                                                                          \right)$. It can be easily computed that $\gamma(T)=\frac{1}{2}$, which is not an extreme point of $W(T)=[0,1]$, but $m(T)=0$, which is an extreme point of $W(T)$.
\end{rem}

\begin{prop}\label{modbddtransformrelns}
Let $T\in \mathcal C(H_1,H_2)$ and $F_T$ be the bounded transform of $T$. Then
\begin{enumerate}
\item \label{rmmbddtransform}$\gamma(F_T)=\dfrac{\gamma(T)}{\sqrt{1+\gamma(T)^2}}$
\item \label{mmboundedtransform}$m(F_T)=\dfrac{m(T)}{\sqrt{1+m(T)^2}}$.
\end{enumerate}
\end{prop}
\begin{proof}
Proof of (\ref{rmmbddtransform}):
In view of Proposition \ref{rmmformula} it  is enough to show that $\gamma(F_T^*F_T)=\dfrac{\gamma(T^*T)}{1+\gamma(T^*T)}$. First we note that $F_T^*F_T=T^*T(I+T^*T)^{-1}=I-(I+T^*T)^{-1}$. Using the formula in Proposition \ref{rmmselfadjoint},
  we get
\begin{align*}
\gamma(F_T^*F_T)&=d(0, \sigma(F_T^*F_T)\setminus {\{0}\})\\
                                &=\inf{\{\dfrac{\mu}{1+\mu}:\mu\in \sigma(T^*T)\setminus{\{0}\}}\}\\
                                &=\inf{\{1-\dfrac{1}{1+\mu}:\mu\in \sigma(T^*T)\setminus{\{0}\}}\}\\
                                &=1-\sup{\{\dfrac{1}{1+\mu}:\mu\in \sigma(T^*T)\setminus{\{0}\}}\}\\
                                &=1-\dfrac{1}{1+\inf{\{\mu:\mu\in \sigma(T^*T)\setminus{\{0}\} }\}}\\
                                &=\dfrac{\gamma(T^*T)}{1+\gamma(T^*T)}.
\end{align*}
Hence we can conclude that $\gamma(F_T)=\dfrac{\gamma(T)}{\sqrt{1+\gamma(T)^2}}$.

Proof of (\ref{mmboundedtransform}): To prove this we need to use  (\ref{minmodformula}) of Lemma \ref{minmodselfadj} and follow the similar steps as above.
\end{proof}

\begin{prop}\label{boundedtransformredminattaining}
Let $T\in \mathcal C(H_1,H_2)$. Then $T\in \Gamma_c(H_1,H_2)$ if and only if $F_T\in \Gamma(H_1,H_2)$.
\end{prop}
\begin{proof}

In view of Proposition \ref{redmingram}, it suffices to show that $T^*T\in \Gamma_c(H_1)$ if and only if $F_T^*F_T\in\Gamma(H_1)$. First, note that $(F_T)^*=F_{T^*}$.  If $T^*T\in \Gamma_c(H_1)$, there exists $x_0\in S_{C(T^*T)}$ such that $T^*Tx_0=\gamma(T^*T)x_0$. Then we have
\begin{align*}
F_T^*F_Tx_0&=\left( I-(I+T^*T)^{-1}\right)x_0\\
                      &=\dfrac{\gamma(T^*T)}{1+\gamma(T^*T)}x_0\\
                      &=\dfrac{\gamma(T)^2}{1+\gamma(T)^2}x_0\\
                      &=\gamma(F_T)^2x_0\\
                      &=\gamma(F_T^*F_T)x_0.
\end{align*}
 This shows that $F_T^*F_T\in \Gamma(H_1)$.

   To prove the converse, suppose  $F_T^*F_T\in \Gamma(H_1)$. Then $F_T^*F_T=F_{T^*}F_T=T^*T(I+T^*T)^{-1}\in \Gamma(H_1)$. Thus there exists $x_0\in N((F_T)^*F_T)^{\bot}=N(F_T)^{\bot}=N(T)^{\bot}$  such that $T^*T(I+T^*T)^{-1}x_0=\gamma(F_{T^*}F_T)x_0$. By Proposition \ref{modbddtransformrelns}, we can obtain that
\begin{equation*}
(I-(I+T^*T)^{-1})(x_0)=\big(1-(\dfrac{1}{1+\gamma(T)^2})\big)(x_0).
\end{equation*}
Equivalently, $(I+T^*T)^{-1}(x_0)=\dfrac{1}{1+\gamma(T)^2}x_0$. That is $x_0\in R\big((I+T^*T)^{-1}\big)=D(I+T^*T)=D(T^*T)$. It follows that $(I+T^*T)(x_0)=(1+\gamma(T)^2)(x_0)$ or $T^*Tx_0=\gamma(T^*T)x_0$, concluding $T^*T$ attains its reduced minimum and so is $T$.
\end{proof}

Next, we would like to prove a Lindenstrauss type theorem for the class of reduced minimum attaining operators. We need the following results for this purpose.
\begin{theorem}\label{specialgapformula} \cite[Theorem 3.1]{shkgrminattaining} Let $S,T\in \mathcal C(H_1,H_2)$ and $D(S)=D(T)$. Then
\begin{enumerate}
\item the operators
$\widehat T^\frac{1}{2}(T-S) \check S^\frac{1}{2}$ and $\widehat S^\frac{1}{2}(T-S) \check T^\frac{1}{2}$  are bounded and
\begin{equation*}
  \theta(S,T)=\max \; \Big\{\| \widehat T^\frac{1}{2}(T-S) \check S^\frac{1}{2}\|,\; \| \widehat S^\frac{1}{2}(T-S) \check T^\frac{1}{2}\| \Big\}
  \end{equation*}

\item if $T-S$ is bounded, then $\theta(S,T)\leq \|S-T\|$.
\end{enumerate}
\end{theorem}

\begin{theorem}\label{minattainingdense}\cite[Theorem 3.5]{shkgrminattaining}
Let $T\in \mathcal C(H_1,H_2)$.  Then for each $\epsilon >0$, there exists  $S\in \mathcal B(H_1,H_2)$ with $\|S\|\leq \epsilon$ such that $S+T$ is minimum attaining and  $\theta(S+T,T)\leq \epsilon$. More over, if $m(T)>0$, then we can choose $S$ to be  a rank one operator.
\end{theorem}

\begin{theorem}\label{rmmdensegeneral}
  Let $T\in \mathcal C(H_1,H_2)$ be densely defined. Then for each $\epsilon>0$ there exists $S\in \mathcal B(H_1,H_2)$ such that
  \begin{enumerate}
  \item $\|S\|\leq \epsilon$
  \item $N(T)=N(T+S)$ and
  \item $T+S$ attains reduced minimum.
  \end{enumerate}
  Moreover, if $\gamma(T)>0$, then we can choose $S$ to be a rank one operator.
\end{theorem}

\begin{proof}
First assume that $\gamma(T)>0$. Consider $T_C:=T|_{C(T)}:N(T)^{\bot}\rightarrow H_2$ is densely defined closed operator.
We may assume that $0<\epsilon<\gamma(T)$. By Theorem \ref{minattainingdense}, there exists $S_0\in \mathcal B(N(T)^{\bot},H_2)$ such that $\|S_0\|\leq \epsilon$ and $T_C+S_0$ is minimum attaining. That is, there exists $x_0\in D(T_C+S_0)=D(T_C)=C(T)$ such that $\|x_0\|=1$ and $\|(T_C+S_0)(x_0)\|=m(T_C+S_0)$. As $m(T_C)=\gamma(T)>0$, we can choose $S_0$ to be a rank one operator.

For $x=u+v\in H_1$ with $u\in N(T),\, v\in N(T)^{\bot}$, define $Sx=S_0v$. Then $\|Sx\|=\|S_0v\|\leq \|v\|\leq \epsilon \|x\|$. Thus $\|S\|\leq \epsilon$. Note that $S$ is a rank one operator.

We claim that $T+S$ attains reduced minimum. Note that $D(T+S)=D(T)$. Let $u\in N(T)$. Then $(T+S)(u)=Tu+Su=0$. Thus $N(T)\subseteq N(T+S)={\{x\in D(T):Tx+Sx=0}\}$. Suppose that $x\in D(T)\setminus N(T)$. Let $x=u+v$ with $u\in N(T)$ and $v\in N(T)^{\bot}$. Then $v\neq 0$. Also, $v\in C(T)$ as $x,u\in D(T)$.
Then
\begin{align*}
\|(T+S)(x)\|&=\|Tv+Sv\|\\
                      &=\|Tv+S_0v\|\\
                      &\geq \|Tv\|-\|S_0v\|\\
                      &\geq \gamma(T)\|v\|-\|S_0\|\|v\|\\
                      &\geq (\gamma(T)-\epsilon)\|v\|\\
                      &>0.
\end{align*}
Thus $(T+S)(x)\neq 0$. Thus $x\notin N(T+S)$. This show that $N(T)=N(T+S)$. Since $D(T+S)=D(T)$, we have $C(T+S)=C(T)$ and hence
\begin{align*}
\gamma(T+S)&=\inf{\{\|(T+S)(x)\|:x\in C(T),\; \|x\|=1}\}\\
                          &=\inf{\{\|(T_C+S_0)(x)\|:x\in C(T),\; \|x\|=1}\}\\
                          &=\|(T_C+S_0)(x_0)\|\\
                          &=\|(T+S)(x_0)\|.
\end{align*}

Next suppose that $\gamma(T)=0$. Let $\epsilon>0$. Choose $x_0\in C(T)$ such that $\|x_0\|=1$ and $\|Tx_0\|<\dfrac{\epsilon}{4}$. Then
\begin{equation*}
\|(T+\dfrac{\epsilon}{2}I)(x_0)\|\geq \dfrac{\epsilon}{2}-\|Tx_0\|\geq \dfrac{\epsilon}{4}.
\end{equation*}
Hence
\begin{equation*}
0<\dfrac{\epsilon}{4}\leq m(T+\dfrac{\epsilon}{2}I)\leq \gamma(T+\dfrac{\epsilon}{2}).
\end{equation*}

By above argument, there exists $\tilde{S}\in \mathcal B(H_1,H_2)$ such that $\|\tilde{S}\|\leq \dfrac{\epsilon}{2}$ and $T+\dfrac{\epsilon}{2}I+\tilde{S}$ attains reduced minimum. Then $\|\dfrac{\epsilon}{2}I+\tilde{S}\|\leq \epsilon$. Take $S=\dfrac{\epsilon}{2}+\tilde{S}$. Then $S$ satisfies all the stated conditions.
\end{proof}

 We have the following consequences.
\begin{theorem}
The following statements holds true;
\begin{enumerate}
\item \label{rmmdensegapmetric}  $\Gamma_c(H_1,H_2)$  is dense in $\mathcal C(H_1,H_2)$ with respect to the  gap metric $\theta(\cdot,\cdot)$.
\item \label{rmmdensecgt}  $\Gamma_c(H_1,H_2)$  is dense in $\mathcal C(H_1,H_2)$ with respect to the metric $\eta(\cdot,\cdot)$
\item \label{clrangedensegap} the set of all closed range operators of $\mathcal C(H_1,H_2)$ is dense in $\mathcal C(H_1,H_2)$ with respect to the metric $\theta(\cdot,\cdot)$
\item \label{clrangedensecgt} the set of all closed range operators of $\mathcal C(H_1,H_2)$ is dense in $\mathcal C(H_1,H_2)$ with respect to the metric $\eta(\cdot,\cdot)$.
\end{enumerate}
\end{theorem}

\begin{proof}
Proof of (\ref{rmmdensegapmetric}): Follows by Theorem \ref{rmmdensegeneral}.

Proof of (\ref{rmmdensecgt}): Let $\epsilon>0$. Then by Theorem \ref{rmmdensegeneral}, we can obtain $S\in \mathcal B(H_1,H_2)$  with $\|S\|\leq \epsilon$ such that $N(T)=N(T+S)$ and $\theta(T,T+S)\leq \epsilon$. By Theorem \ref{equivalenceofcgtandgap}, it follows that $\eta(T+S,T)=\theta(T+S,T)\leq \epsilon$. Hence the claim.

Proof of  (\ref{clrangedensegap}): Since $\Gamma_c(H_1,H_2)$ is a subset of the set of all closed range operators in $\mathcal C(H_1,H_2)$,  the conclusion is immediate by (\ref{rmmdensecgt}) above.

Proof of (\ref{clrangedensecgt}): This follows by  Proposition \ref{closedrange} and (\ref{rmmdensecgt}) above.
\end{proof}

Using the equivalence of the gap metric and the metric induced by the operator norm on $\mathcal B(H_1,H_2)$ we can obtain the following consequences.
\begin{cor}
The following statements are true.
\begin{enumerate}
\item\label{rmmbdddensecgt} $\Gamma(H_1,H_2)$ is dense in $\mathcal B(H_1,H_2)$ with respect to the operator norm
\item \label{clrangedensenorm} the set of all bounded closed range operators is dense in $\mathcal B(H_1,H_2)$ with respect to the operator norm.
\end{enumerate}
\end{cor}

\section{A Corrected  Formula }
In \cite{shkgrcgt} an incorrect formula was given for the gap $\theta(T,nI)$ between a symmetric, closed  densely defined operator $T$ and $nI$. In  this section we point out the error and give a correct formula with proof.

  Let $T\in \mathcal C(H_1,H_2)$. Recall that $\Check {T}:=(I+T^*T)^{-1}$ and $\widehat T:=(I+TT^*)^{-1}$. First we recall the incorrect formula given in \cite{shkgrcgt}.

Let $T\in \mathcal C(H)$ be symmetric. Let $n\in \mathbb N$ be fixed. Then
\begin{equation*}
\theta (T,nI)=\dfrac{1}{\sqrt{1+n^2}}.
\end{equation*}

The correct formula is given in the following Proposition.
\begin{prop}
Let $T\in \mathcal C(H)$ be symmetric. Let $n\in \mathbb N$ be fixed. Then
\begin{equation*}
\theta(T,nI)=\dfrac{1}{\sqrt{1+n^2}}\max \left\{\|(T-nI)\Check{T}^\frac{1}{2}\|,\; \|(T^*-nI)\widehat{T}^\frac{1}{2}\|\right\}.
\end{equation*}
In particular, if $T=T^*$, then
\begin{equation*}
\theta(T,nI)=\dfrac{\|(T-nI)(I+T^2)^\frac{-1}{2}\|}{\sqrt{1+n^2}}.
\end{equation*}
Further more, if $-\dfrac{1}{n}\in \sigma(T)$, then $\theta(T,nI)=1$.
\end{prop}
\begin{proof}
We use the formula in Theorem \ref{shkgrgapformula}. Let $S=nI$. Then $\Check{S}^{\frac{1}{2}}=\dfrac{I}{\sqrt{1+n^2}}=\Hat S^\frac{1}{2}$ and $S\Check{S}^\frac{1}{2}=\dfrac{nI}{\sqrt{1+n^2}}$.
Now
\begin{equation*}
S\Check{S}^\frac{1}{2}\Check{T}^\frac{1}{2}-\widehat{S}^\frac{1}{2}T\Check{T}^\frac{1}{2}=\dfrac{1}{\sqrt{1+n^2}}(n\Check{T}^\frac{1}{2}-T\Check{T}^\frac{1}{2})=\dfrac{1}{\sqrt{1+n^2}}(nI-T)\Check{T}^\frac{1}{2}.
\end{equation*}
And
\begin{equation*}
T\Check{T}^\frac{1}{2}\Check{S}^\frac{1}{2}-\widehat{T}^\frac{1}{2}S\Check{S}^\frac{1}{2}=\dfrac{1}{\sqrt{1+n^2}}(T\Check{T}^\frac{1}{2}-n\widehat{T}^\frac{1}{2}).
\end{equation*}
Let $A:=n\Check{T}^\frac{1}{2}-T\Check{T}^\frac{1}{2}$ and $B:=T\Check{T}^\frac{1}{2}-n\widehat{T}^\frac{1}{2}$. Then $\theta(T,nI)=\dfrac{1}{\sqrt{1+n^2}}\max{\{\|A\|,\|B\|}\}$.
Note that $B^*=T^*\widehat{T}^\frac{1}{2}-n\widehat{T}^\frac{1}{2}=(T^*-nI)\widehat{T}^\frac{1}{2}$. Since $\|B^*\|=\|B\|$, we get that
\begin{equation*}
\theta (T,nI)=\dfrac{1}{\sqrt{1+n^2}}\max \left\{\|(T-nI)\Check{T}^\frac{1}{2}\|,\; \|(T^*-nI)\widehat{T}^\frac{1}{2}\|\right\}.
\end{equation*}

If $T=T^*$, then $A=B$ and hence the formula follows in this case. As $A^*=A$ and $A$ is bounded, we have
\begin{equation*}
\|A\|=\sup{\{|\lambda|:\lambda \in \sigma(A)}\}=\sup \left\{\dfrac{|n-\lambda|}{\sqrt{1+\lambda^2}}:\lambda \in \sigma(T)\right\}.
\end{equation*}

 Hence consider the function
\begin{equation*}
f(x)=\dfrac{|x-n|}{\sqrt{1+x^2}},\;  x\in \sigma(T)\subseteq \mathbb R.
\end{equation*}
If  $x_0=\dfrac{-1}{n}\in \sigma(T)$, then we have $f(x_0)=\sqrt{1+n^2}$ and hence $\|A\|\geq \sqrt{1+n^2}$. Hence $\theta (T,nI)=1$.
\end{proof}
The following example illustrates the formula.
\begin{eg}
Let $H=\ell^2$ and $\mathcal D={\{(x_m)\in H: (mx_m)\in H}\}$. Define $T: \mathcal D\rightarrow H$ by
\begin{equation*}
T(x_1,x_2,x_3,\dots)=(x_1,2x_2,3x_3,\dots )\;\; \text{for all}\;\; (x_m)\in \mathcal D.
\end{equation*}
Clearly $T$ is densely defined, $T=T^*$ and range of $T$ is closed. Let ${\{e_m:m\in \mathbb N}\}$ be the standard orthonormal basis of $H$. Then $Te_m=me_m$ for each $m\in \mathbb N$. Hence $\mathbb N\subseteq \sigma_p(T)$, the point spectrum of $T$. In fact, we can show that $\sigma(T)=\mathbb N$.

 \end{eg}

 For each $m\in \mathbb N$, we have
\begin{align*}		
T^2e_m&=m^2e_m\\
(I+T^2)(e_m)&=(1+m^2)e_m\\
(I+T^2)^\frac{1}{2}e_m&=\sqrt{1+m^2}e_m\\
\Check{T}^\frac{1}{2}e_m=(I+T^2)^\frac{-1}{2}e_m&=\dfrac{1}{\sqrt{1+m^2}}e_m\\
T\Check{T}^\frac{1}{2}e_m&=\dfrac{m}{\sqrt{1+m^2}}e_m.
\end{align*}
Now $(T-nI)\Check{T}^\frac{1}{2}\;e_m=\dfrac{m-n}{\sqrt{1+m^2}}\;e_m$ for each $m\in \mathbb N$. Hence
\begin{equation}\label{generalformula}
\|(T-nI)\Check{T}^\frac{1}{2}\|=\sup \left \{\dfrac{|m-n|}{\sqrt{1+m^2}}:m\in \mathbb N\right\}=\max \left\{1, \dfrac{|n-1|}{\sqrt{2}}\right\}.
 \end{equation}	
 Hence $\theta(T,nI)=\dfrac {1}{\sqrt{1+n^2}}\,\max \left\{1,\dfrac{|n-1|}{\sqrt{2}}\right\}$.
\begin{note}
For a fixed $n\in \mathbb N$, the sequence $a_m:=\left\{\dfrac{|m-n|}{\sqrt{1+m^2}}\right\}$ decreases for $m=1$ to $n$ ($a_n=0$) and then increases with $\displaystyle \lim_{m\rightarrow \infty}a_m=1$.
\end{note}


\end{document}